\documentclass{elsarticle}

\usepackage[centertags]{amsmath}
\usepackage{amsfonts}
\usepackage{amssymb}
\usepackage{newlfont}

\newtheorem{theorem}{Theorem}

\newtheorem{problem}{Problem}

\newtheorem{definition}{Definition}
\newtheorem{example}{Example}

\newenvironment{proof}{\medskip{\em Proof.}}{\par}

\newfont{\gothic}{eufm9 scaled 1200}

\begin{document}

\begin{frontmatter}
\title{Remarks on Haar meager sets and Haar null sets in spaces of sequences}
\author{Eliza  Jab{\l}o{\'n}ska}
\ead{elizapie@prz.edu.pl}

\address{Department of Mathematics, Rzesz{\'o}w University of
Technology, Powsta\'{n}c\'{o}w~Warszawy~12, 35-959~Rzesz{\'o}w, POLAND}


\begin{abstract}
In the paper we will show how to construct a Haar meager set (consequently meager) which is not Haar null, and conversely, a meager Haar null set which is not Haar meager in spaces of sequences: $l_p$ with $p\geq1$, $c_0$ or $c$. It refers to the paper \cite{Darji}.\end{abstract}
\end{frontmatter}


\section{Introduction}

In 1972 J.P.R.~Christensen defined "Haar null" sets in an abelian Polish group (a topological abelian group with a complete separable
metric) in such a way that in a locally compact group it is equivalent to the notion of Haar measure zero set. More precisely, in a fixed abelian Polish group $X$ a set $A\subset X$ is called \textit{Haar null} if there is a Borel probability measure $\mu$ on $X$ and a Borel set $B\subset X$ such that $A\subset B$ and $\mu(x+B)=0$ for all $x\in X$.
These definition has been extended further to nonabelian
groups by J.~Mycielski \cite{Myc}.
Unaware of the result of Christensen, B.R.~Hunt, T.~Sauer and J.A.~Yorke \cite{HSY}-\cite{HSY1} found this notation again, but in a topological abelian group with a complete metric (not necessary separable).

In 2013 U.B.~Darji introduced another family of "small" sets in an abelian Polish group, which is equivalent to the notion of meager sets in a locally compact group. In an abelian Polish group $X$ he called a set $A\subset X$ \textit{Haar meager} if there is a Borel set $B\subset X$ with $A\subset B$, a compact metric space $K$ and a continuous function $f:K\to X$ such that $f^{-1}(B+x)$ is meager in $K$ for all $x\in X$.

The main aim of the paper is to show easy constructions of a Haar meager, but not Haar null set and, conversely, a meager Haar null set which is not Haar meager in spaces of sequences: $l_p$ with $p\geq1$, $c_0$ or $c$.

\section{The main results}

\begin{definition}
{\em Let $X$ be an abelian Polish group, $\mathcal{B}(X)$ be the Borel $\sigma$--algebra on $X$ and denote by $\mathcal{F}(X)$ the family of all sets $A\subset X$ such that
$$
\forall_{\;K\subset X\mbox{\small{-compact}}}\;\; \exists _{\;x_K\in X}\;\;K+x_K\subset A.
$$}
\end{definition}

In fact $\mathcal{F}(X)$ is a proper linearly invariant $\sigma$--filter. What is interesting,
each set $A\in \mathcal{F}(X)\cap \mathcal{B}(X)$ is neither Haar null (in view of the Ulam theorem), nor Haar meager.

S. Solecki \cite{Solecki}, and also E. Matou\u{s}kov\'{a} and M. Zelen\'{y} \cite{Matouse}, showed how to find a closed nowhere dense set from the family $\mathcal{F}$ in any abelian non-locally compact Polish group. We use this fact to construct a Haar meager, but not Haar null set, as well as a meager Haar null set which is not Haar meager in spaces of sequences: $l_p$ with $p\geq1$, $c_0$ or $c$.\\

First we prove two theorems, which we will use in further considerations.

\begin{theorem}\label{2}
Let $X$, $Y$ be an abelian Polish group. If $A\subset X$ is Haar meager and $B\subset B_0$ for some $B_0\in\mathcal{B}(Y)$, then  the set $A\times B\subset X\times Y$ is Haar meager.
\end{theorem}

\begin{proof}
Assume that $A\subset X$ is Haar meager in $X$, i.e. there are a set $A_0\in\mathcal{B}(X)$ with $A\subset A_0$, a compact metric space $K$ and a continuous function $f:K\to X$ such that $f^{-1}(A_0+x)$ is meager in $K$ for each $x\in X$. Take any compact set $L\subset Y$ and define a continuous function $g:K\times L\to X\times Y$ in the following way:
$$
g(k,l)=(f(k),l)\;\;\mbox{for every}\;\;(k,l)\in K\times L.
$$
Then,
$$
\begin{array}{r}
g^{-1}((A_0\times B_0)+(x,y))=g^{-1}((A_0+x)\times (B_0+y))\;\;\;\;\;\;\\[1ex]
=f^{-1}(A_0+x)\times [(B_0+y)\cap L]
\end{array}$$
for each $(x,y)\in X\times Y$. Since the set $f^{-1}(A_0+x)$ is meager in $K$, by the Kuratowki-Ulam theorem the set $g^{-1}((A_0\times B_0)+(x,y))$ is meager in $K\times L$. Clearly, $A\times B\subset A_0\times B_0$ and $A_0\times B_0\in\mathcal{B}(X\times Y)$, what ends the proof.
\end{proof}

\begin{theorem}\label{1}
Let $X$, $Y$ be an abelian Polish group. For every set $A\in \mathcal{B}(X)\cap \mathcal{F}(X)$ and non-Haar meager $B\in \mathcal{B}(Y)$ the set $A\times B\subset X\times Y$ is not Haar meager.
\end{theorem}

\begin{proof}
Clearly, $A\times B\in\mathcal{B}(X\times Y)$. Take a compact metric space $K$ and a continuous function $f:K\to X\times Y$. Then there are continuous functions $f_X:K \to X$ and $f_Y:K\to Y$ such that
$$f(z)=(f_X(z),f_Y(z))\;\;\mbox{for each}\;\;z\in K.$$
 The set $f_X(K)$ is compact in $X$ and $A\in\mathcal{F}(X)$, so $A\supset f_X(K)+x_K$ for some $x_K\in X$. Hence $f_X^{-1}(A-x_K)\supset K$. Since $B$ is not Haar meager in $Y$, $f_Y^{-1}(B+y_K)$ is comeager in $K$ for some $y_K\in Y$. Moreover:
$$
\begin{array}{l}
f^{-1}((A\times B)+(-x_k,y_K))=f^{-1}((A-x_K)\times (B+y_K))\\[1ex]
\;\;\;\;\;\;\;\;\;\;\;\;\;\;\;\;\;=f_X^{-1}(A-x_K)\cap f_Y^{-1}(B+y_K)\supset K\cap f_Y^{-1}(B+y_K).
\end{array}
$$
Thus $f^{-1}((A\times B)+(-x_K,y_K))$ is comeager in $K$ and, consequently,  the set $A\times B$ is not Haar meager in $X\times Y$.
\end{proof}
\vspace{0.2cm}

The above theorem suggest the following

\begin{problem}
{\em Let $X$, $Y$ be an abelian Polish group. Is it rue or false that for every non-Haar meager sets $A\in\mathcal{B}(X)$ and $B\in\mathcal{B}(Y)$ the set $A\times B\subset X\times Y$ is not Haar meager?}
\end{problem}

A negative answer implies the same question under additional assumption that one of abelian Polish group $X$, $Y$ is locally compact.

\section{Applications}

Now, consider the space $X$ as one of the following spaces of sequences: $c_0$, $c$ or $l_p$ with $p\geq1$. Such spaces have a very nice property:  $X=\mathbb{R}\times X$.
\vspace{0.5cm}

 Fix $A\in \mathcal{B}(X)\cap\mathcal{F}(X)$. Let $S:=B\times A\subset \mathbb{R}\times X$, where $B\subset \mathbb{R}$ is a meager set of positive Lebesgue measure.

By the analogue of the Fubini theorem \cite[Theorem~6]{Ch} it is easy to observe that $S$ is not Haar null, since $S(a)=B$ is the set of the positive Lebesgue measure for each $a\in A$ and $A$ is not Haar null. Moreover, in view of Theorem~\ref{2}, the set $S$ is Haar meager.
In this way we constructed the set $S$, which is \textbf{Haar meager} (consequently \textbf{meager}), but \textbf{not Haar null} in $X$.
\vspace{0.5cm}

Now, fix $A\in \mathcal{B}(X)\cap\mathcal{F}(X)$ once again. Let $T:=C\times A\subset \mathbb{R}\times X$, where $C\subset \mathbb{R}$ is a comeager set of the Lebesgue measure zero.

By an analogue of the Fubini theorem \cite[Theorem~6]{Ch} we can easy deduce that $T$ is Haar null in $X$, because $T(a)=C$ is the set of the Lebesgue measure zero for each $a\in A$ and $T(a)=\emptyset$ for each $a\in X\setminus A$. Moreover, by Theorem~\ref{1} the set $T$ is not Haar meager in $\mathbb{R}\times X=X$. Finally, $T$ is meager according to the Kuratowski-Ulam theorem, because the set $A$ is meager. In this way we constructed the set $T$, which is \textbf{Haar null}, \textbf{meager}, but \textbf{not Haar meager} in $X$.

\begin{example}{\em
Define the set $A=\{(x_n)_{n\in\mathbb{N}}\in c_0: \;\forall_{n\in\mathbb{N}}\;x_n\geq0\}$. Such set is a closed nowhere dense set (see \cite[Example 6.2]{Mat}), which belongs to the filter $\mathcal{F}(c_0)$.

Let $S:=B\times A\subset c_0$, where $B\subset \mathbb{R}$ is a meager set of the positive Lebesgue measure. Then the set $S$ is Haar meager, but not Haar null in $c_0$.

Let $T:=C\times A\subset c_0$, where $C\subset \mathbb{R}$ is a comeager set of the Lebesgue measure zero. Such set $T$ is Haar null, meager, but not Haar meager in $c_0$.}
\end{example}

\begin{problem}
{\em Let $X$ be any abelian Polish group, which is not locally compact. How to find a Haar meager but not Haar null set and, conversely, how to construct a Haar null, meager but not Haar meager set in $X$?}
\end{problem}

\end{document}